\documentclass[12pt]{article}

\newtheorem{theorem}{Theorem}

\newenvironment{proof}{\begin{trivlist}
    \item[\hskip\labelsep{\it Proof.}]}{$\hfill\Box$\end{trivlist}}

\newcommand{\RR}{{\Bbb R}}
\newcommand{\NN}{{\Bbb N}}

\newcommand{\wal}{{\rm wal}}

\newcommand{\rd}{\,{\rm d}}

\usepackage{latexsym,amsfonts,amsmath,amssymb}
             
\title{Lebesgue constants for the Walsh system and the discrepancy of the van der Corput sequence}
\author{Josef Dick and Friedrich Pillichshammer}
\date{}

\begin{document}
\maketitle

\begin{abstract}
In this short note we report on a coincidence of two mathematical quantities that, at first glance, have little to do with each other. On the one hand, there are the Lebesgue constants of the Walsh function system that play an important role in approximation theory, and on the other hand there is the star discrepancy of the van der Corput sequence that plays a prominent role in uniform distribution theory. Over the decades, these two quantities have been examined in great detail independently of each other and important results have been proven. Work in these areas has been carried out independently, but as we show here, they actually coincide. Interestingly, many theorems have been discovered in both areas independently, but some results have only been known in one area but not in the other.
\end{abstract}

\paragraph{Lebesgue Constants.} For a given system of orthonormal functions $\Phi=\{\phi_k \ : \ k =0,1,\ldots\}$ in $L_2([a,b])$ the $n$-th Lebesgue function is defined as 
\begin{equation*}
L_n^{\Phi}(u):=\int_a^b \left|\sum_{k=0}^{n-1} \phi_k(x) \phi_k(u) \right| \rd x\qquad \mbox{for $u \in [a,b]$.}
\end{equation*}
It follows easily from the fact $S_n(f)(u)=\int_a^b \left(\sum_{k=0}^{n-1} \phi_k(x) \phi_k(u)\right) f(u)\rd u$ that 
\begin{equation*}
L_n^{\Phi}(u)=\sup_{\|f\|_{C([a,b])} \le 1} |S_n(f)(u)|\qquad \mbox{for $u \in [a,b]$,}
\end{equation*}
where $\|\cdot \|_{C([a,b])}$ is the uniform norm and $S_n(f)(u)=\sum_{k=0}^{n-1} \widehat{f}(k) \phi_k(u)$ is the $n$-th partial sum of the Fourier expansion of $f$ with respect to the system $\Phi$. If the $n$-th Lebesgue function $L_n^{\Phi}$ is constant over $[a,b]$, then its function value is called the $n$-th Lebesgue constant, which is denoted by $L_n^{\Phi}$ or simply by $L_n$. Lebesgue functions/constants are a fundamental tool in approximation theory; see, e.g., \cite{AT,KS,SWS,zyg}.

In this note we consider as an orthonormal system the Walsh functions and report a connection of the corresponding Lebesgue functions to the star discrepancy of the van der Corput sequence, which is an important quantity in uniform distribution and discrepancy theory. Both Lebesgue functions for the Walsh system and star discrepancy of van der Corput sequence are well-studied objects in literature, but so far it has not been observed that both quantities coincide. This is, at least for us, a very interesting connection between concepts from different branches of mathematics that allows to transfer results for one measure to the other and vice versa.

\paragraph{Walsh functions.} For a nonnegative integer $k$ with base $2$ representation $k=k_0+k_1 2+\cdots+k_m 2^m$ the $k$-th {\it Walsh function} $\wal_k:\RR \rightarrow \RR$, periodic with period one, is defined by 
\begin{eqnarray}
\wal_k(x):=(-1)^{x_1 k_0 +x_2 k_1+\cdots + x_{m+1} k_m},\nonumber
\end{eqnarray}
when $x \in [0,1)$ has (canonical) base $2$ representation $x=\frac{x_1}{2}+\frac{x_2}{2^2}+\cdots$. The system $\mathcal{W}=\{\wal_k \ : \ k=0,1,\ldots \}$ is a complete orthonormal system in $L_2([0,1])$. For information about Walsh functions see \cite{fine,SWS} or \cite[Appendix~A]{DP10}.

In \cite[Section~5]{fine} Fine introduced the Walsh Dirichlet kernel $$D_n(x,u):= \sum_{k=0}^{n-1} \wal_k(x) \wal_k(u)$$ and the so-called Lebesgue functions $$L_n(x):=\int_0^1|D_n(x,u)| \rd u$$ (for simplicity we write $L_n(x)$ rather than $L_n^{\mathcal{W}}(x)$) for the Walsh systems and proved several interesting results. In the first place he proved that the $L_n(x)$ are in fact independent of $x$ (see \cite[p.386]{fine}) and for this reason they coincide with their Lebesgue constants $L_n$. In \cite[Theorem~IX]{fine} several properties of the Lebesgue constants $L_n$ are summarized.

\begin{theorem}[Fine]\label{thm:fine}
The Lebesgue constants $L_n$ of the Walsh system satisfy 
\begin{equation}\label{fo:LCW}
L_n=\nu-\sum_{1 \le j < i \le \nu} 2^{n_i-n_{j}},
\end{equation}
where $n$ is of the form $n=2^{n_1}+2^{n_2}+\cdots+2^{n_{\nu}}$ with integer exponents $n_1 > n_2 > \cdots > n_{\nu}\ge 0$. Furthermore, the following properties hold:
\begin{enumerate}
\item $L_{2n}=L_n$ and $L_{2n+1}=(1+L_n+L_{n+1})/2$.
\item $L_n=O(\log n)$.
\item $\frac{1}{n} \sum_{k=1}^n L_k =\frac{\log n}{4 \log 2}+O(1)$.
\item $\limsup_{n \rightarrow \infty} \left[L_n-\left(\frac{4}{9}+\frac{\log 3}{3 \log 2}+\frac{\log n}{3 \log 2}\right)\right]=0$.
\item the generating function for $L_n$ is $\sum_{n=1}^{\infty} L_n z^n =\frac{1}{2} \frac{z}{(1-z)^2} \sum_{k=0}^{\infty}\frac{1}{2^k} \frac{1-z^{2^k}}{1+z^{2^k}}$ for $|z|<1$.
\end{enumerate}
\end{theorem}

Further results on the Lebesgue constants for the Walsh system are shown in \cite{AS1,AS2}. 

\paragraph{Star discrepancy.}
For a sequence $X=(x_k)_{k \ge 0}$ in $[0,1)$ the $n$-th discrepancy function is defined as $$\Delta_{X,n}(t) :=\frac{\#\{k \in \NN_0\ : \ k < n,\ x_k \in [0,t)\}}{n}-t \qquad \mbox{for $t \in [0,1]$}.$$ The star discrepancy of the initial $n$ terms of $X$ is then defined as $$D_n^*(X):=\sup_{t \in [0,1]}|\Delta_{X,n}(t)|.$$ This is a quantitative measure for the irregularity of distribution of the sequence $X$ in $[0,1)$. The sequence $X$ is uniformly distributed in the sense of Weyl \cite{weyl} iff $\lim_{n \rightarrow \infty}D_n^{\ast}(X)=0$. The smaller $D_n^{\ast}(X)$, the better the $n$ initial terms of $X$ are distributed in $[0,1]$. For an introduction to uniform distribution and discrepancy, see, e.g., \cite{kuinie}.

\paragraph{Star discrepancy of the van der Corput sequence.}
The prototype of a uniformly distributed sequence is the van der Corput sequence (see \cite{vdc35} or \cite{FKP}) whose construction is based on the reflection of the binary digits of nonnegative integers $k$ around the binary point. If $k$ has binary expansion $k=k_0+k_1 2+k_2 2^2+\cdots$ (which is of course finite) with $k_j \in \{0,1\}$ for $j=0,1,2,\ldots$, then the $k$-th element $y_k$ of the van der Corput sequence $Y^{{\rm vdC}}=(y_k)_{k \ge 0}$ is given by 
\begin{equation*}
y_k :=  \frac{k_0}{2}+\frac{k_1}{2^2}+\frac{k_2}{2^3}+\cdots.
\end{equation*}

The star discrepancy of the van der Corput sequence is intensively studied in literature, see, e.g., \cite{befa,befa77,dlp04,fau1990,hab1966,spi}. See also the survey \cite{FKP} for many other references.

Consider the star discrepancy $D_n^{\ast}(Y^{{\rm vdC}})$ of the first $n$ terms of the van der Corput sequence and set $d_n:=n D_n^{\ast}(Y^{{\rm vdC}})$ to be the non-normalized star discrepancy. 

\begin{theorem}\label{thm2}
The $n$-th Lebesgue constant $L_n$ for the Walsh system is exactly the nonnormalized star discrepancy $d_n$ of the first $n$ terms of the van der Corput sequence, that is $$d_n=L_n.$$ 
\end{theorem}

\begin{proof}
Let $n \in \NN$ be of the form $n=2^{n_1}+2^{n_2}+\cdots+2^{n_{\nu}}$ with integer exponents $n_1 > n_2 > \cdots > n_{\nu}\ge 0$. For the initial $n$ terms of $Y^{\rm vdC}$ we have 
\begin{align*}
\{y_0,y_1,\ldots,y_{n-1}\} = & \bigcup_{i=1}^{\nu}\left\{y_{\ell}\ : \ \ell \in \{2^{n_1}+\cdots+2^{n_{i-1}},\ldots, 2^{n_1}+\cdots+2^{n_{i-1}}+2^{n_i}-1\}\right\}\\
= & \bigcup_{i=1}^{\nu} \left\{\frac{k}{2^{n_i}}+\frac{1}{2^{n_{i-1}+1}}+\frac{1}{2^{n_{i-2}+1}}+\cdots +\frac{1}{2^{n_1+1}} \ : \ k \in \{0,1,\ldots,2^{n_i}-1\}\right\},
\end{align*}
where we put $2^{n_1}+\cdots+2^{n_{i-1}}:=0$ if $i=1$.

It is well known that the star discrepancy of the van der Corput sequence is nonnegative (see \cite[Remark~3]{p04}) and twice the $L_1([0,1])$-norm of the discrepancy function (see, for example, \cite{PA}). Hence 
\begin{align}\label{stl1}
d_n = & 2 n  \int_0^1 \Delta_{Y^{\rm vdc},n}(t) \rd t = 2 \int_0^1 \sum_{\ell=0}^{n-1} (\boldsymbol{1}_{[0,t)}(y_{\ell}) -t) \rd t\\
= & 2 \sum_{\ell=0}^{n-1} \left(\frac{1}{2}-y_{\ell}\right) = 2 \sum_{i=1}^{\nu} \sum_{k=0}^{2^{n_i}-1} \left(\frac{1}{2}-\frac{k}{2^{n_i}}- \sum_{j=1}^{i-1} \frac{1}{2^{n_j+1}}\right)\nonumber\\
= & 2 \left(\frac{\nu}{2}-\sum_{i=1}^{\nu}2^{n_i} \sum_{j=1}^{i-1} \frac{1}{2^{n_j+1}} \right) = \nu-\sum_{1\le j < i \le \nu}2^{n_i-n_j} =L_n,\nonumber
\end{align}
where the last equality is \eqref{fo:LCW} from Theorem~\ref{thm:fine}.
\end{proof}

The representation $d_n=\nu-\sum_{1\le j < i \le \nu}2^{n_i-n_j}$ for the star discrepancy of $Y^{\rm vdC}$ is also mentioned (in an equivalent form) in \cite[p.~61]{spi}.

\paragraph{Applications.} Theorem~\ref{thm2} can now be used to transfer results for the star discrepancy of the van der Corput sequence to results for the Lebesgue constants for the Walsh system and vice versa. In particular Items~1-4 from Theorem~\ref{thm:fine} by Fine have also been proven in the context of discrepancy of the van der Corput sequence by B\'{e}jian and Faure \cite{befa,befa77}. Therein, the estimate $d_n \le \frac{\log n}{3 \log 2} +1$ from \cite[Th\'eor\`eme~3]{befa77} yields a refinement of item~{\it 2} in Theorem~\ref{thm:fine}. Furthermore, it is known (see \cite[Lemme~4.1]{fau1990}) that for every $r \in \mathbb{N}$ we have $$\max_{n \in [2^{r-1},2^r]} d_n = \frac{r}{3}+\frac{7}{9}+\frac{(-1)^r}{9 \cdot 2^{r-1}}$$ and the maximum is attained for $n=(2^{r+1}+(-1)^r)/3$.  Thus, for all $n$ of this form we have $d_n \ge \frac{\log n}{3 \log 2} +O(1)$. This result has been proven (in a slightly weaker form) for the Lebesgue constants for the Walsh system in \cite[Theorem~2]{AS2}.

From Theorem~\ref{thm2} we also obtain new representations of the Lebesgue constant or star discrepancy of the van der Corput sequence that have not been known so far. For instance, the definition of $L_n$ leads to a new representation of $D_n^{\ast}(Y^{{\rm vdC}})$ via 
\begin{equation*}
D_n^{\ast}(Y^{{\rm vdC}})=\frac{d_n}{n}=\frac{L_n}{n}=\int_0^1 \left|\frac{1}{n}\sum_{k=0}^{n-1} \wal_k(x)\right| \rd x.
\end{equation*}
Using this formula, we can derive another new formula for the star discrepancy of the van der Corput sequence. We have
\begin{align*}
D_n^{\ast}(Y^{{\rm vdC}}) = & \sum_{m=0}^{2^{n_1+1}-1} \int_{m 2^{-n_1-1}}^{(m+1) 2^{-n_1-1}} \left|\frac{1}{n} \sum_{k=0}^{n-1} \wal_k(x) \right| \rd x.
\end{align*}
For $x \in [m 2^{-n_1-1}, (m+1) 2^{-n_1-1})$ we can write $\wal_k(x) = \wal_{m'}(y_k)$, where $m'=m'(m)=m_{n_1}+m_{n_1-1} 2 +\cdots+m_0 2^{n_1}$ for $m$ with binary expansion $m=m_0+m_1 2+\cdots+m_{n_1} 2^{n_1}$. Then
\begin{align*}
D^\ast_n(Y^{{\rm vdC}}) = & \frac{1}{2^{n_1+1}} \sum_{m'=0}^{2^{n_1+1}-1}  \left| \frac{1}{n} \sum_{k=0}^{n-1} \wal_{m'}(y_k) \right|.
\end{align*}

On the other hand, a well known representation for the star discrepancy of the van der Corput sequence according to \cite[Th\'eor\`eme~1]{befa77} leads to a new representation for $L_n$ via 
$$L_n=d_n=n D_n^{\ast}(Y^{{\rm vdC}}) =\sum_{r=1}^{\infty} \left\|\frac{n}{2^r}\right\|  = \sum_{r=1}^m
\left\|\frac{n}{2^r}\right\|+\frac{n}{2^m}\
 \ \ \mbox{ whenever $1 \le n \le 2^m$},$$
where $\|x\|=\min(\{x\},1-\{x\})$ is the distance of a real $x$ to the nearest integer. This formula immediately implies the recursion in Item~1 of Theorem~\ref{thm:fine} (note that $d_1=L_1=1$), which was already known for $d_n$ before (see~\cite[p.~13-07]{befa77}). 

Also the generating function for $L_n$ (see Item~5 in Theorem~\ref{thm:fine}) was unknown in terms of discrepancy, which can now be formulated as  $$\sum_{n=1}^{\infty} d_n z^n =\frac{1}{2} \frac{z}{(1-z)^2} \sum_{k=0}^{\infty}\frac{1}{2^k} \frac{1-z^{2^k}}{1+z^{2^k}} \qquad \mbox{for $|z|<1$.}$$

For the star discrepancy of the van der Corput sequence we know a central limit theorem from \cite{dlp04}, which may now be formulated in terms of Lebesgue constants for the Walsh system. Accordingly, for every real $y$ and for $N \rightarrow \infty$ we have
$$\frac{1}{N} \# \left\{ n < N \ : \  L_n \le \frac{\log n}{4 \log 2} + \frac{y}{4} \sqrt{
\frac{\log n}{3 \log 2}} \right\} = \Phi(y)+ o(1),$$ where 
$\Phi(y) = \frac1{\sqrt{2\pi}} \int_{-\infty}^y \exp(-t^2/2) \rd t $ is the Gaussian cumulative distribution function. 
That is, the Lebesgue constants for the Walsh system satisfy a central limit theorem.

As the final example of this note we reformulate another result for the Lebesgue constants to obtain a so far unknown result for the star discrepancy of the van der Corput sequence. For $t \in [0,1]$ and $m \in \mathbb{N}$ let $n_t(m):= \lfloor 2^m(1+t)\rfloor$. Then \cite[Theorem~5]{AS2} in terms of discrepancy of the van der Corput sequence states:
\begin{enumerate}
\item For almost all $t \in [0,1]$ we have $$\lim_{m \rightarrow \infty} \frac{d_{n_t(m)}}{\log n_t(m) }=\frac{1}{4 \log 2}.$$
\item For all dyadic rational $t \in [0,1]$ we have $$\lim_{m \rightarrow \infty} \frac{d_{n_t(m)}}{\log n_t(m)}=0.$$
\item There exists a dense subset $A \subseteq [0,1]$ such that $$\liminf_{m \rightarrow \infty} \frac{d_{n_t(m)}}{\log n_t(m)}=0 \quad \mbox{ and }\quad \limsup_{m \rightarrow \infty} \frac{d_{n_t(m)}}{\log n_t(m)}=\frac{1}{3 \log 2}$$ for all $t \in A$.
\end{enumerate}




\begin{thebibliography}{10}

\bibitem{AS1} S.V.~Astashkin, E.M.~Semenov: Lebesgue constants of a Walsh system. (Russian) Dokl. Akad. Nauk 462(5): 509--511, 2015; translation in Dokl. Math. 91(3): 344–346, 2015.

\bibitem{AS2} S.V.~Astashkin, E.M.~Semenov: Lebesgue constants of a Walsh system and Banach limits. (Russian)
Sibirsk. Mat. Zh. 57(3): 512--526, 2016; translation in Sib. Math. J. 57(3): 398--410, 2016.

\bibitem{AT} A.P.~Austin, L.N.~Trefethen: Trigonometric interpolation and quadrature in perturbed points. SIAM J. Numer. Anal. 55(5): 2113--2122, 2017.

\bibitem{befa} R.~B\'{e}jian, H.~Faure: Discr\'{e}pance de la suite de van der Corput. (French) C. R. Acad. Sci., Paris, S\'{e}r. A 285: 313--316, 1977.

\bibitem{befa77} R.~B\'{e}jian, H.~Faure: Discr\'{e}pance de la suite de van der Corput. (French) S\'{e}minaire Delange-Pisot-Poitou (Th\'{e}orie des nombres) 13: 1--14,1977/78.


\bibitem{DP10} J.~Dick, F. Pillichshammer: Digital Nets and Sequences. Discrepancy Theory and Quasi-Monte Carlo Integration. Cambridge University Press, Cambridge, 2010.


\bibitem{dlp04} M.~Drmota, G.~Larcher, F.~Pillichshammer: Precise distribution properties of the van der Corput sequence and related sequences. Manuscripta Math. 118: 11--41, 2005.


\bibitem{fau1990} H.~Faure: Discr\'{e}pance quadratique de la suite de van der Corput et de sa sym\'{e}trique. (French) Acta Arith. 55: 333--350, 1990.

\bibitem{FKP} H.~Faure, P.~Kritzer, F.~Pillichshammer: From van der Corput to modern constructions of sequences for quasi-Monte Carlo rules. Indag. Math. 26: 760--822, 2015.

\bibitem{fine} N.J.~Fine: On the Walsh functions. Trans. Amer. Math. Soc. 65: 372--414, 1949.


\bibitem{hab1966} S.~Haber: On a sequence of points of interest for numerical quadrature. J. Res. Nat. Bur. Standards Sect. B70: 127--136, 1966.

\bibitem{KS} S.~Kaczmarz, H.~Steinhaus: Theorie der Orthogonalreihen. (German) Monografie Matematyczne, Chelsea Publishing Company, New York, 1951. 


\bibitem{kuinie} L.~Kuipers, H.~Niederreiter: Uniform Distribution of Sequences. John Wiley, New York, 1974.


\bibitem{p04} F.~Pillichshammer: On the discrepancy of $(0,1)$-sequences. J. Number Theory 104: 301--314, 2004.




\bibitem{PA} P.D.~Pro\u{\i}nov, E.Y.~Atanassov: On the distribution of the van der Corput generalized sequences. C. R. Acad. Sci. Paris S\'er. I Math. 307(18): 895--900, 1988.



\bibitem{SWS} F.~Schipp, W.R.~Wade, P.~Simon: Walsh Series. An Introduction to Dyadic Harmonic Analysis. Adam Hilger Ltd., Bristol, 1990. 

\bibitem{spi} L.~Spiegelhofer: Discrepancy results for the van der Corput sequence. Unif. Distrib. Theory 13(2): 57--69, 2018.

\bibitem{vdc35} J.G.~van der Corput: Verteilungsfunktionen I-II. Proc. Akad. Amsterdam 38: 813--821, 1058--1066, 1935.

\bibitem{weyl} H.~Weyl: \"Uber die Gleichverteilung mod. Eins. (German) Math. Ann. 77: 313--352, 1916. 

\bibitem{zyg} A.~Zygmund: Trigonometric Series. Cambridge University Press, New York, 1959.


\end{thebibliography}
\end{document}